\documentclass[a4]{article}
\usepackage{amsmath,amsthm,amssymb,amscd}
\usepackage{ascmac}
\usepackage[arrow,matrix]{xy}
\usepackage{enumerate}
\usepackage[dvipdfmx]{graphicx}

\newtheorem{thm}{Theorem}%[section]
\newtheorem{hyp}{Hypothesis}
\newtheorem{property}{Property}

\newtheorem{dfn}{Definition}[section]

\newtheorem{prop}[dfn]{Proposition}

\newtheorem{lem}[dfn]{Lemma}

\newtheorem{rem}[dfn]{Remark}
\newtheorem{fact}[dfn]{Fact}

\def\dim{\mathop{\mathrm{dim}}\nolimits}

\def\pr{\mathop{\mathrm{Pr}}\nolimits}

\def\real{\mathop{\mathrm{Re}}\nolimits}
\def\imag{\mathop{\mathrm{Im}}\nolimits}
\def\diag{\mathop{\mathrm{diag}}\nolimits}

\def\clos{\mathop{\mathrm{cl}}\nolimits}

\begin{document}

\title{The absolute spectrum revisited from a topological viewpoint}
%\thanks{This work was supported by the Japan Science and Technology Agency, CREST.}} 

\author{Ayuki Sekisaka\\
Meiji Institute for Advanced Study of\\
 Mathematical Sciences, Meiji University,\\
8F High-Rise Wing, Nakano\\
4-21-1 Nakano, Nakano-ku, Tokyo, Japan.\\
E-mail:sekisaka@meiji.ac.jp}
\date{}

\maketitle

{\bf keywords }absolute spectrum; essential spectrum; topological method

{\bf AMS }34, 37

\thispagestyle{plain}
\markboth{AYUKI SEKISAKA}{Absolute spectrum from topological viewpoint}

\begin{abstract}
We consider the topological relation behind the spectral behavior of a linear operator that arises in the stability problem of traveling waves on a large bounded domain. When the domain size tends to infinity, 
the absolute and asymptotic essential spectra appears as accumulation sets of eigenvalues under separated and periodic boundary conditions, respectively. 
We present new proofs of Sandstede and Scheel [Theorems 4 and 5 of \cite{SS2}] in a topological framework.
The eigenfunction induces a curve on the Grassmannian manifold.
To extract topological information from them,
we decompose the Grassmannian into the submanifolds using the Schubert cycles,
and analyze the curves on each submanifolds.
%Finally,
%we give an example of the accumulation behaviors of eigenvalues for general class of reaction-diffusion systems. 
\end{abstract}

%%%%%%%%%%%%%%%%%%%%%%%%%%%%%%%%%%%%%%%%%%%%%%%%%%%%%%%%%%%%%%%%%
\section{Introduction}
We first consider the stability problem of 
traveling wave solutions for reaction-diffusion systems on the bounded interval
$I_\ell :=[-\ell, \ell]$ with the boundary conditions
\begin{eqnarray}
u_t = B u_{xx} + F(u), \; x \in I_\ell, t > 0. 
\end{eqnarray}
Let $B$ be a diagonal positive matrix, $u \in \mathbb{R}^N$ 
and $F: \mathbb{R}^N \rightarrow \mathbb{R}^N$ be a smooth function. 
The eigenvalue problem for a traveling wave solution $u(x,t)=\hat{u}(\xi)$ 
comes from the linearization
\begin{eqnarray}\label{EQ:1}
L_\ell p =  B p_{\xi \xi} + c p_\xi + D_u F(\hat{u}(\xi)) p = \lambda p
\end{eqnarray}
in the moving frame of $\xi = x - ct$ where $c$ is the wave speed. 
We rewrite the eigenvalue problem (\ref{EQ:1}) 
as a first-order system of $Y \in \mathbb{C}^{2N}$ 
\begin{eqnarray}\label{EQ:2}
\displaystyle{
\begin{pmatrix} p \\ q \end{pmatrix}' = 
\begin{pmatrix} 0 & 1 \\ 
B^{-1} (\lambda - \partial_u F(\hat{u}(\xi))) & -B^{-1} c   
\end{pmatrix} 
\begin{pmatrix} p \\ q \end{pmatrix} 
}
& \rightarrow &
Y' = A(\xi; \lambda) Y, 
\end{eqnarray} 
where $\displaystyle{ ' := \frac{\rm d}{{\rm d} \xi} }$ and $A$ is a $2N \times 2N$ matrix.
We reformulate the spectral problem of the family of linear operators 
with respect to the parameter $\lambda$ on the bounded interval $I_\ell$:
\begin{eqnarray}\label{EQ:3}
{\cal T}_\ell (\lambda): H^1_{bd}(I_\ell, \mathbb{C}^{2N})
\rightarrow L^2 (I_\ell, \mathbb{C}^{2N});
Y \mapsto \frac{d Y}{d \xi} - A(\xi; \lambda) Y
\end{eqnarray}
for the Sobolev space $H^1$ of $L^2$ functions where {\em bd} 
indicates two classes of boundary conditions, i.e., 
separated or periodic boundary conditions. We note that $\lim \limits_{\xi \rightarrow \pm \infty}A(\xi; \lambda) = A_{\pm}(\lambda)$ on the real line. 
It is well known that the spectrum of the linear operator ${\cal T}_\ell(\lambda)$ consists solely of eigenvalues on the bounded domain, 
whereas the essential spectrum appears on the unbounded domain. 
What has attracted mathematical interest is the manner in which the spectral behavior changes when the unbounded domain is truncated to a bounded interval depending on the imposed boundary conditions.  

For the stability problem of traveling wave solutions to reaction-diffusion systems, 
Sandstede and Scheel \cite{SS2} introduced the so-called {\em absolute spectrum} as the accumulation sets of eigenvalues when the domain size $2 \ell$ becomes infinite under the separated boundary conditions. 
For the periodic boundary condition,
the absolute spectra are embedded in the boundary of the essential spectra as $\ell \rightarrow \infty$. 
Sandstede and Scheel also investigated the accumulation of eigenvalues arising from the problem in which two unstable front solutions form one stable {\em glued} pulse solution. 
The absolute spectra play a crucial role in determining the instability of the pulse solution when the pulse width becomes infinite.  
%%%%

The spectral behavior of the linear operator has also been discussed in relation to the topological properties of the relevant manifold. 
Alexander, Gardner and Jones \cite{AGJ} shed light on some geometrical aspects of the stability problem of traveling wave solutions on the unbounded domain,
and later examined stability over the bounded domain \cite{GJ}.
Introducing a stability index as the first Chern number of a complex vector bundle on a sphere $S^{2}$ (i.e. homeomorphic to a projective space $\mathbb{CP}^1$), 
they provided the topological structure behind Evans function theory,
in which the eigenvalues of the operator ${\cal T}_\ell(\lambda)$ coincide with zeros of a certain analytic function constructed from the analytic basis spanning the solution manifolds of (\ref{EQ:2}) \cite{Evans4}. 
On the bounded domain,
the eigenvalues accumulate into the essential spectrum as $\ell \to \infty$ if the coefficient matrix satisfies the periodic conditions $A(\xi + 2 \ell;\lambda) = A(\xi;\lambda)$ \cite{RG1}. 
Nii \cite{SN1} gave a topological viewpoint of the accumulation behavior of eigenvalues and also proposed another approach to the stability problem of glued pulse solutions.
However,
the scope of his approach based on the complex line bundle was limited to relatively low- dimensional ordinary differential equations (ODEs),
corresponding to the specific reaction-diffusion systems of the FitzHugh--Nagumo equations.

The present study has its roots in the important work described above.
Most of the remainder of this paper is devoted to further proofs of two theorems given by Sandstede and Scheel [Theorem 4 on p.261 and 5 on p.265 of \cite{SS2}], 
related to the accumulation behavior of the eigenvalues of the linear operator in (\ref{EQ:3}). 
The proofs presented by Sandstede and Scheel used analytical techniques, 
i.e., the construction of a generalized Evans function to introduce exponential weights to the functional space,
and bifurcation analysis using the Lyapunov--Schmidt reduction.
In fact, the entire phase space of (\ref{EQ:2}) is divided 
into two subspaces, the $k$-dimensional unstable subspace $U_-$ of $A_-$ 
and the $(2N-k)$-dimensional stable subspace $U_+$ of $A_+$. 
The absolute spectrum for (\ref{EQ:3}) is then defined as 
\begin{eqnarray}
\Sigma_{abs} := \{ \lambda \in \mathbb{C} | \real \mu_\pm^k(\lambda)
=  \real \mu_\pm^{k+1}(\lambda) \}, 
\end{eqnarray}
where we label the eigenvalues  
of $A_\pm(\lambda)$, $\mu_\pm^j(\lambda)$ for 
$j = 1, \cdots, 2N$, according to their real part as 
$\real \mu_\pm^j(\lambda) \ge \real \mu_\pm^{j+1}(\lambda)$. 
We define the map ${\cal G}_{\ell}: \Sigma_{abs} \to 
G_{k}(\mathbb{C}^{2N})$ on the complex Grassmanian $G_k (\mathbb{C}^{2N})$, 
corresponding to the reformulated system. 
The existence problem for the eigenvalues of ${\cal T}_\ell(\lambda)$ turns out to be equivalent to the existence problem for the one-dimensional curve 
${\cal G}_{\ell}(\Sigma_{abs})$. 
Hence, $\lambda$ is an eigenvalue of $L_\ell$ if and only if the curve ${\cal G}_{\ell}(\Sigma_{abs})$ intersects with $U_{+}$ on $G_{k}(\mathbb{C}^{N})$. 
 
First, we show that the curve ${\cal G}_{\ell}(\Sigma_{abs})$ is diffeomorphic to a curve on the one-dimensional Schubert cycle on $S^2 = \mathbb{CP}^1 \subset G_{N}(\mathbb{C}^{2N})$. 
Next, we show that $U_+$ coincides with the one-codimensional Schubert cycle,
and that $U_+$ intersects ${\cal G}_{\ell}(\lambda)$ transversally for sufficiently large $\ell$.
As $\lambda$ varies,
the topological measurement of the curve ${\cal G}_{\ell}(\lambda)$ corresponding to the winding number changes, and at this point an eigenvalue appears in the original problem (\ref{EQ:1}). 
Finally, we give topological proofs for two theorems of Sandstede and Scheel \cite{SS2}, 
in which the eigenvalues accumulate on the absolute spectrum if and only if the number of intersections between ${\cal G}_{\ell}(\lambda)$ and $U_{+}$ becomes infinite as $\ell \rightarrow \infty$.

Although a large number of studies have been made on stability problems of traveling waves for multi-component reaction-diffusion systems,
little is known about the topological relation behind the spectral behavior of the linear operator and dynamical systems that arise in the stability problem of traveling waves.
We expect that the topological results of the stability problem of traveling wave solutions for the FitzHugh--Nagumo equations (e.g., \cite{SN2}, \cite{SN3}, \cite{SN1}) can be extended to multi-component reaction-diffusion systems.

The remainder of this paper is organized as follows. In Section 2, we describe the structural hypotheses needed to prove the theorems in a topological framework. 
Some settings used in \cite{SS2} are rewritten from a geometrical viewpoint.  
In Section 3,
we address the main subject of this paper, 
proving theorems on the accumulation behavior of eigenvalues of the linear operator ${\cal T}_{\ell}(\lambda)$. 
%Finally, we supply an example of the absolute spectra for the linear scalar diffusion--advection equation,
%and extend this to the problem for general reaction--diffusion systems.

%%%%%%%%%%%%%%%%%%%%%%%%%%%%%%%%%%%%%%%%%%%%%%%%%%%%%%%%%%%%%%%%%%%%%%%%
\section{Structural hypothesis}\label{sec:1}
We consider an abstract system of partial differential equations of the form
\begin{equation*}\label{eq:int2}
u_{t}=N(\partial_{x},u).
\end{equation*}
The behavior of infinitesimal perturbations of the traveling wave $u(\xi)$ with a moving frame $\xi=x-ct$ is determined by eigenvalues of the linear operator:
\begin{equation*}
L_{\ell}=\partial_{u}N(\frac{d}{d \xi},u) \Bigr{|} _{u=u(\xi)} + c \frac{d}{d \xi}.
\end{equation*}
We assume that the eigenvalue problem $L_{\ell}p=\lambda p$ can be rewritten as a system of first-order ODEs.
We obtain the following system on $\mathbb{C}^{N}$:
\begin{equation}\label{eq:1}
Y'=A(\xi;\lambda)Y,
\end{equation}
where $N$ depends on the dimension of $u$ and the higher-order derivatives in $L_{\ell}$.
It is convenient to pose the eigenvalue problem in the general form of \eqref{eq:1} rather than to begin with the partial differential equations,
because several types of equations will lead to matrices $A(\xi;\lambda)$ with somewhat different forms.
For example,
problems involving traveling waves for reaction-diffusion systems,
the generalized KdV equation, and the Boussinesq equations can be rewritten in the form of \eqref{eq:1} \cite{SS2}, \cite{RG2}.

We treat the first-order system \eqref{eq:1} as the family of linear operators with two classes of boundary conditions.
The correct function space for separated boundary conditions is given by
\begin{equation}
H_{sep}^{1}(I_{\ell},\mathbb{C}^{2N}):=H^{1}(I_{\ell},\mathbb{C}^{2N}) \cap \{  Y \mid Y(-\ell) \in U_{-}, \ Y(\ell) \in U_{+} \},
\end{equation}
and we consider the linear operator
\begin{equation*}
{\cal T}_{\ell}(\lambda):H_{sep}^{1}(I_{\ell},\mathbb{C}^{2N}) \to L^{2}(I_{\ell},\mathbb{C}^{2N});\ 
{\cal T}_{\ell}(\lambda)Y=\frac{d}{d x} Y - A(x;\lambda)Y.
\end{equation*}
Note that $U_{\pm}$ can be realized by several boundary conditions that are induced by boundary operators $B_{\pm}$.
However, it cannot be realized by periodic boundary conditions.

Similarly,
the function space for periodic boundary conditions is given by
\begin{equation}
H_{per}^{1}(I_{\ell},\mathbb{C}^{2N}):=H^{1}(I_{\ell},\mathbb{C}^{2N}) \cap \{  Y \mid Y(-\ell) = Y(\ell) \}
\end{equation}
and we consider the linear operator
\begin{equation*}
{\cal T}^{per}_{\ell}(\lambda):H_{per}^{2}(I_{\ell},\mathbb{C}^{2N}) \to L^{2}(I_{\ell},\mathbb{C}^{2N});\ 
{\cal T}^{per}_{\ell}(\lambda)Y=\frac{d}{d x} Y - A(x;\lambda)Y.
\end{equation*}
Throughout this paper,
we assume that $A(x;\lambda) \in M_{N}(\mathbb{C})$ is smooth in $x$ and analytic in $\lambda$.
Under the above derivation,
our main interest is the accumulation of eigenvalues of ${\cal T}_{\ell}$ (resp. ${\cal T}^{per}_{\ell}$) with the parameter $\lambda$ in the B-spectrum (see \cite{LO}).
\begin{dfn}
We say that $\lambda$ is in the spectrum $\Sigma$ of ${\cal T}_{\ell}$ if ${\cal T}_{\ell}(\lambda)$ is not invertible.
$\lambda$ is in the point spectrum $\Sigma_{pt}$ of ${\cal T}_{\ell}$ if ${\cal T}_{\ell}(\lambda)$ is a Fredholm operator with index zero.
The complement $\Sigma_{ess} \setminus \Sigma_{pt} =: \Sigma_{ess}$ is called the essential spectrum.
\end{dfn}

The following fact is well known.
\begin{fact} \cite{SS1}.
The operators ${\cal T}_{\ell}(\lambda)$ (resp. ${\cal T}^{per}_{\ell}$) on the bounded interval $(-\ell,\ell)$ with separated boundary conditions (resp. periodic boundary conditions) are Fredholm with index zero for all $\lambda$.
\end{fact}

Consequently, $\Sigma$ consists of only eigenvalues.
In \cite{SS1},
Sandstede and Scheel proved that eigenvalues of ${\cal T}_{\ell}$ accumulate in a specific curve,
the so-called absolute spectrum.
We may characterize the absolute spectrum of ${\cal T}_{\ell}$ topologically for large values of $\ell$.

\subsection{Case 1: Separated boundary condition}\label{sec:2}
First,
we assume that $A(x;\lambda)$ has the asymptotic matrices on a bounded, open domain $\Omega \subset \mathbb{C}$.
\begin{hyp}\label{hyp:1}
For any $\lambda \in \Omega$, $A(x;\lambda)$ has the following property
We assume that $A(x;\lambda)$ is locally constant outside a compact interval $[ -\ell_{0}, \ell_{0}]$,
i.e.,
\begin{eqnarray*}
A(x;\lambda)=\left\{
\begin{array}{l}
A_{-}(\lambda), \text{ for } x \leq -\ell_{0} \\
A_{+}(\lambda), \text{ for } x \geq \ell_{0}
\end{array}
\right.
\end{eqnarray*}
where $A_{\pm}(\lambda)$ depend analytically on $\lambda \in \Omega$.
\end{hyp}

We assume the boundary subspaces $U_{\pm}$ satisfy the following hypothesis.
\begin{hyp}\label{hyp:2}
\begin{equation}
\dim U_{-}=:i_{-}, \ \dim U_{+}=N-i_{-} =:i_{+},
\end{equation}
and $i_{-} \leq i_{+}$.
\end{hyp}

We label the eigenvalues $\mu_{\pm}^{j}(\lambda)$ of $A_{\pm}(\lambda)$ according to their real part,
and repeat them according to their multiplicity,
i.e.,
\begin{eqnarray*}
&&\real \mu_{\pm}^{j}(\lambda) \geq \real \mu_{\pm}^{j+1}(\lambda), \ 1 \leq j \leq N-1.
\end{eqnarray*}

We define then the absolute spectrum for ${\cal T}_{\ell}$ as follows.
\begin{dfn}\label{dfn:1}(The absolute spectrum, \cite{SS2}).
$\Sigma_{abs}^{\Omega,+} \subset \Omega$ consists of $\lambda \in \Omega$ satisfying $\real \mu_{+}^{i_{-}}(\lambda) = \real \mu_{+}^{i_{-}}(\lambda)$.
Analogously,
$\Sigma_{abs}^{\Omega,-} \subset \Omega$ consists of $\lambda \in \Omega$ satisfying $\real \mu_{-}^{i_{+}}(\lambda) = \real \mu_{-}^{i_{+}}(\lambda)$.
Then the absolute spectrum $\Sigma_{abs}^{\Omega,\pm}$ is given by
\begin{equation}
\Sigma_{abs}^{\Omega}=\Sigma_{abs}^{\Omega,+} \cup \Sigma_{abs}^{\Omega,-}.
\end{equation}
\end{dfn}

The absolute spectrum is not the spectrum for ${\cal T}_{\ell}$.
However,
eigenvalues of ${\cal T}_{\ell}$ accumulate on this subset when $\ell \to \infty$.
Roughly speaking,
it is the set of accumulation points of eigenvalues with respect to ${\cal T}_{\ell}$.

Another key hypothesis concerns a generic property of eigenvalues of $A_{\pm}(\lambda)$ on the absolute spectrum.
In particular,
the absolute spectrum induces the curve in the Grassmannian manifold via the flow induced by \eqref{eq:1}.
Hence,
this generic property gives several properties to the induced curve.
\begin{dfn}(Non-degenerate absolute spectrum).
The subset $S_{abs}^{\Omega} \subset \Sigma_{abs}^{\Omega}$ is defined as follows.
It is dense in $\Sigma_{abs}^{\Omega}$,
and $\lambda_{*} \in S_{abs}^{\Omega}$ satisfies the following conditions.
\begin{equation}\label{gen:2}
\real \mu_{\pm}^{i_{\mp}-1}(\lambda_{*}) > \real \mu_{\pm}^{i_{\mp}}(\lambda_{*}) = \real \mu_{\pm}^{i_{\mp}+1}(\lambda_{*}) > \real \mu_{\pm}^{i_{\mp}+2}(\lambda_{*}),
\end{equation}
$\mu_{\pm}^{i_{\mp}}(\lambda_{*}) \neq \mu_{\pm}^{i_{\mp} + 1}(\lambda_{*})$,
and % for any $r \in \mathbb{R}$,
\begin{equation}\label{gen:3}
\frac{d}{d\lambda}\left( \mu_{\pm}^{i_{\mp}}(\lambda) - \mu_{\pm}^{i_{\mp}+1}(\lambda) \right) \Big{|}_{\lambda = \lambda_{*}} \neq 0.  % i r.
\end{equation}
Moreover,
for any $\lambda_{c} \in S_{abs}^{\Omega}$,
there exists $\delta >0$ such that
\begin{eqnarray*}
\frac{d}{d\lambda}\left( \mu_{\pm}^{i_{\mp}}(\lambda) - \mu_{\pm}^{i_{\mp}+1}(\lambda) \right) \Big{|}_{\lambda = \lambda_{*}} \neq  i r,
\end{eqnarray*}
holds for each $\lambda_{*} \in B(\lambda_{c};\delta) \setminus S_{abs}^{\Omega}$ and any $r \in \mathbb{R}$,
where $B(\lambda_{c};\delta)$ is a $\delta$-ball centered at $\lambda_{c}$.
\end{dfn}

Note that the set $S_{abs}^{\Omega}$ consists of curve segments.
In \cite{SS1},
$S_{abs}^{\Omega}$ is called the reducible absolute spectrum.
However, we call $S_{abs}^{\Omega}$ the non-degenerate absolute spectrum to emphasize the generic property of eigenvalues.
By this definition of the non-degenerate absolute spectrum $S_{abs}^{\Omega}$,
we can take small $\delta>0$ such that $B(\lambda_{c};\delta)\setminus S_{abs}^{\Omega}$ consists of two half disks $B_{1}$ and $B_{2}$ with the following properties.
\begin{property}\label{property}
\begin{equation*}
\real \mu_{+}^{i_{-}-1}(\lambda) > \real \mu_{+}^{i_{-}}(\lambda) > \real \mu_{+}^{i_{-}+1}(\lambda) > \real \mu_{+}^{i_{-}+2}(\lambda),
\end{equation*}
for any $\lambda \in B_{1}$,
and we fix the order,
\begin{equation*}
\real \mu_{+}^{i_{-}-1}(\lambda) > \real \mu_{+}^{i_{-} +1}(\lambda) > \real \mu_{+}^{i_{-}}(\lambda) > \real \mu_{+}^{i_{-}+2}(\lambda),
\end{equation*}
for any $\lambda \in B_{2}$.
\end{property}
Throughout this paper,
we always take $B(\lambda_{c};\delta)$ satisfying the above properties.

We assume that $S_{abs}^{\Omega}$ exists in $\Omega$.
\begin{hyp}\label{hyp:4}
$S_{abs}^{\Omega} \neq \emptyset$.
\end{hyp}

We obtain the topological information from the dynamics induced by the condition $\real \mu_{\pm}^{i_{\mp}}(\lambda) = \real \mu_{\pm}^{i_{\mp}+1}(\lambda)$.
In particular,
the relationship between generalized eigenspaces associated with $\mu_{\pm}^{j}(\lambda), \ (j=1,\cdots,i_{\mp})$ and boundary subspaces $U_{\pm}$ gives one of most important structures.
\begin{hyp}\label{hyp:3}(Transversality).
Let $\bar{E}_{+}(\lambda)$ and $\bar{E}_{-}(\lambda)$ be generalized eigenspaces associated with $\mu^{j}_{+}(\lambda) \ (j=1,\cdots,i_{-}+1)$ and $\mu^{j}_{-}(\lambda) \ (j=1,\cdots,i_{+}+1)$, respectively.
We assume that $U_{+}$ and $\bar{E}_{+}(\lambda)$ are in a general position for any $\lambda \in \Omega$.
Similarly,
$U_{-}$ and $\bar{E}_{-}(\lambda)$ are in general positions for any $\lambda \in \Omega$.
That is,
\begin{equation}\label{gen1}
\bar{E}_{\pm}(\lambda) + U_{\pm}=\mathbb{C}^{N}.
\end{equation}
In addition, $\Phi(x,-\ell;\lambda)U_{-}$ is not contained in $\bar{E}_{+}(\lambda)$ for any $\lambda \in \Omega$ and $\ell \geq x>\ell_{0}$,
and $\Phi(x,\ell;\lambda)U_{+}$ is not contained in $\bar{E}_{-}(\lambda)$ for any $\lambda \in \Omega$ and $\ell \leq x<-\ell_{0}$.
%that is,
%\begin{equation}
%\dim (\Phi(x,-\ell;\lambda)U_{-} \cap \bar{E}_{+}(\lambda)) \geq  k+1.
%\end{equation}
\end{hyp}

\begin{rem}
Under Hypotheses \ref{hyp:2} and \ref{hyp:3},
$\dim U_{-}=i_{-}$,
$\dim U_{+}=N-i_{-}=i_{+}$ and $\dim \bar{E}_{\pm}(\lambda) = i_{\mp} +1$.
Hence,
\begin{equation*}
\dim (\bar{E}_{\pm}(\lambda) \cap U_{\pm})= 1.
\end{equation*}
\end{rem}

\subsection{Case 2: Periodic boundary conditions}\label{sec:3}
Under periodic boundary conditions,
we assume that the asymptotic matrices $A_{\pm}(\lambda)$ are equal to one another.
\begin{hyp}\label{hyp:p1}
We assume that $A(x;\lambda)$ satisfies
\begin{eqnarray*}
A(x;\lambda)=A_{0}(\lambda), \ \text{for } |x| \geq \ell_{0},
\end{eqnarray*}
where $A_{0}(\lambda)$ depends analytically on $\lambda \in \Omega$.
\end{hyp}

Note that linear subspaces in $\mathbb{C}^{N}$ cannot realize periodic boundary conditions for \eqref{eq:1}.
Therefore,
we transform the periodic boundary conditions to the separated boundary conditions using the stability index theory for $\gamma$-eigenvalue problems \cite{RG2}.

Consider equation \eqref{eq:1} with $2N$ additional equations so as to express the periodic conditions as separated boundary conditions,
\begin{eqnarray*}
&&Y'=A(x;\lambda)Y,\\
&&W'=0,
\end{eqnarray*}
or simply
\begin{eqnarray}\label{eq:p1}
\hat{Y}'=\hat{A}(x;\lambda)\hat{Y}.
\end{eqnarray}
In addition,
we set $U_{-}=U_{+}:=\{ (Y,Y) \in \mathbb{C}^{2N}| Y \in \mathbb{C}^{N} \}$.
By the above derivation,
the periodic boundary conditions are transformed to the separated boundary conditions.
We assume that the eigenvalues $\mu_{0}^{j}(\lambda)$ of $A_{0}(\lambda)$ are ordered such that
\begin{equation}
\real \mu_{0}^{j}(\lambda) \geq \real \mu_{0}^{j+1}(\lambda), \ j=1,\cdots,N-1,
\end{equation}
and consider the absolute spectrum for \eqref{eq:p1}. 
For separated boundary conditions,
the definition of the absolute spectrum depends on the dimension of the boundary subspaces $U_{\pm}$,
that is,
$\real \mu^{\dim U_{\mp}}_{\pm}(\lambda) = \real \mu^{\dim U_{\mp}+1}_{\pm}(\lambda)$.
However,
the dimension of the boundary subspaces $U_{\pm}$ for periodic boundary conditions cannot be determined from \eqref{eq:1} because the eigenvalues of $\hat{A}_{0}(\lambda)$ then satisfy
\begin{equation*}
\real \mu_{0}^{1}(\lambda) \geq \cdots \geq \real\mu^{k}_{0}(\lambda) \geq 0 = \cdots = 0 \geq \real\mu_{0}^{k+1}(\lambda) \geq \cdots \geq \real \mu_{0}^{N}(\lambda),
\end{equation*}
and the dimension of $U_{\pm}$ is always equal to $N$.
Therefore,
we define the set of accumulation points of eigenvalues for ${\cal T}_{\ell}^{per}$ as follows.
\begin{dfn}(Extrapolated essential spectral set, \cite{SS2}).
$\lambda_{*} \in D$ is not in the extrapolated essential spectral set $\Sigma_{ext}^{e}$ of the family $\{ {\cal T}_{\ell}^{per} \}_{\ell}$ if there exists $\ell_{*}>0$, $\delta >0$,
and $n \in \mathbb{N}$ such that ${\cal T}_{\ell}^{per}$ has at most $n$ eigenvalues in $B(\lambda_{*};\delta)$ for any $\ell_{*} \geq \ell$.
\end{dfn}
The extrapolated essential spectral set $\Sigma_{ext}^{e}$ was defined in \cite{SS2},
and characterizes the accumulation of eigenvalues in several cases.
We define two algebraic curves known as the asymptotic essential spectrum and the non-degenerate essential spectrum.
\begin{dfn}(Asymptotic essential spectrum \cite{SS2}).
$\lambda_{*}$ is in the asymptotic essential spectrum $\Sigma_{ess}^{\mathbb{R}}$ if $A_{0}(\lambda_{*})$ is not hyperbolic,
that is,
\begin{equation*}
\Sigma_{ess}^{\mathbb{R}} := \{ \lambda \in \mathbb{C} \mid \sigma(A_{0}(\lambda)) \cap i\mathbb{R} \neq \emptyset  \},
\end{equation*}
where $\sigma(A_{0}(\lambda))$ is spectral set of $A_{0}(\lambda)$.
\end{dfn}
Note that $\Sigma_{ess}^{\mathbb{R}}$ is the essential spectrum for ${\cal T}(\lambda)$ which is defined by
\begin{equation*}
{\cal T}(\lambda):D({\cal T})=H^{1}(\mathbb{R},\mathbb{C}^{N})  \to L^{2}(\mathbb{R},\mathbb{C}^{N}); {\cal T}(\lambda)u=\frac{du}{dx} -A(x;\lambda)u,
\end{equation*}
where $A(x;\lambda) \to A_{0}(\lambda)$ as $|x| \to \infty$ exponentially.
\begin{dfn}(Non-degenerate essential spectrum).
The so-called non-degenerate essential spectrum $S_{per}\subset \Sigma_{ess}^{\mathbb{R}}$ is defined as follows:
$S_{per}$ is dense in $\Sigma_{ess}^{\mathbb{R}}$,
and $\lambda \in S_{per}$ satisfies the following conditions.
$A_{0}(\lambda_{*}) \cap i\mathbb{R} = \{ \mu_{0}^{k}(\lambda_{*}) = i\omega^{k}(\lambda_{*}) \}$,
and 
\begin{equation}\label{pgen:1}
\frac{d}{d\lambda}\omega^{k}(\lambda)\big|_{\lambda=\lambda_{*}} \neq 0.  % i r.
\end{equation}
Moreover,
for any $\lambda_{c} \in S_{per}$,
there exists $\delta >0$ such that
\begin{eqnarray*}
\frac{d}{d\lambda} \mu_{0}^{k}(\lambda) \Big{|}_{\lambda = \lambda_{*}} \neq  i r
\end{eqnarray*}
holds for each $\lambda_{*} \in B(\lambda_{c};\delta) \setminus S_{per}$ and any $r \in \mathbb{R}$,
where $B(\lambda_{c};\delta)$ is a $\delta$-ball centered at $\lambda_{c}$.
\end{dfn}
In \cite{SS2},
$S_{per}$ is called the reducible essential spectrum.
However,
we can add to its generic properties.
Therefore,
we can take small $\delta>0$ such that $B(\lambda_{c};\delta)\setminus S_{per}$ consists of two half-disks $B_{1}$ and $B_{2}$ with the following property.
\begin{property}\label{property2}
\begin{equation}
\real \mu_{0}^{k-1}(\lambda)>\real \mu_{0}^{k}(\lambda) > 0 > \real \mu_{0}^{k+1}(\lambda)
\end{equation} 
for any $\lambda \in B_{1}$,
and 
\begin{equation}
\real \mu_{0}^{k-1}(\lambda)> 0 > \real \mu_{0}^{k}(\lambda)  > \real \mu_{0}^{k+1}(\lambda)
\end{equation} 
for any $\lambda \in B_{2}$.
\end{property}

In the case of periodic boundary conditions,
we always take $B(\lambda_{c};\delta)$ satisfying property \ref{property2}. 

We assume that $S_{per}$ exists.
\begin{hyp}\label{hyp:p2}
$S_{per} \neq \emptyset$.
\end{hyp}

Let $\bar{E}_{0}(\lambda)$ be a $N+1$ dimensional generalized eigenspace associated with $\mu_{0}^{k}(\lambda)$,
$0,\cdots,0$,
and $\Phi(x,y;\lambda)$ be a fundamental solution matrix for \eqref{eq:p1}.
We then assume the following transversality of $\bar{E}_{0}(\lambda)$ and $U_{\pm}$.
\begin{hyp}\label{hyp:p3}(Transversality)
For any $\lambda_{c} \in S_{per}$,
there exists $\delta >0$ such that
\begin{equation*}
\bar{E}_{0}(\lambda) + U_{\pm} = \mathbb{C}^{2N}
\end{equation*}
for any $\lambda \in B(\lambda_{c};\delta)$.
Moreover,
$\Phi(\ell,-\ell;\lambda)U_{-}+\bar{E}_{0}(\lambda)= \mathbb{C}^{2N}$ for any $\lambda \in B(\lambda_{c};\delta)$.
\end{hyp}

This hypothesis means that $\dim(\bar{E}_{0}(\lambda) \cap U_{\pm})=1$.

Note that $\gamma$-eigenvalues are defined by the following conditions.
Define the subspace $U_{\gamma}:=\{ (\gamma Y,Y) \mid U \in \mathbb{C}^{N} \}$ where $\gamma \in \mathbb{C}$ satisfies $|\gamma|=1$. 
Then $\lambda$ is $\gamma$-eigenvalue of ${\cal T}_{\ell}^{per}$ if and only if there exists a nontrivial solution ${\hat Y}(x;\lambda)$ of \eqref{eq:p1} satisfying ${\hat Y}(-\ell;\lambda) \in U_{1}, \ {\hat Y}(\ell;\lambda) \in U_{\gamma}$.
We can consider the case of the $1$-eigenvalue of ${\cal T}_{\ell}^{per}$.

%%%%%%%%%%%%%%%%%%%%%%%%%%%%%%%%%%%%%%%%%%%%%%%%%%%%%%%%%%%%%%%%%%%%%%%%%%%%
\section{}
\subsection{Main results}\label{sec:4}

Under the above conditions,
eigenvalues accumulate on the absolute spectrum associated with separated boundary conditions,
and accumulate on the asymptotic essential spectrum associated with periodic boundary conditions.

\begin{thm}\label{thm:1}(Case of separated boundary conditions).
Let $B(\lambda_{c};\delta)$ be a $\delta$-ball centered at a $\lambda_{c} \in S_{abs}^{\Omega}$.
Hypotheses \ref{hyp:1}-\ref{hyp:3} are satisfied.
Then, for any $n \in \mathbb{N}$ and $\delta > 0$,
there exists $\ell_{*} > 0$ such that the family $\{ {\cal T}_{\ell} \}$ has at least $n$ eigenvalues in $B(\lambda_{c};\delta)$ for any $\ell \geq \ell_{*}$.
\end{thm}

\begin{thm}\label{thm:2}(Case of periodic boundary conditions).
Hypotheses \ref{hyp:p1}-\ref{hyp:p3} are satisfied.
Then, $\Sigma_{ext}^{e} = S_{per}$.
\end{thm}

Theorem \ref{thm:1} (resp. Theorem \ref{thm:2}) holds if we replace the asymptotic matrices $A_{\pm}(\lambda)$ with periodic matrices $A_{\pm}(x +\ell_{\pm};\lambda)=A_{\pm}(x;\lambda)$ (resp. $A_{0}(x +\ell_{per};\lambda)=A_{0}(x;\lambda)$).
In such a case,
$\Sigma_{abs}^{D,\pm}$ (resp. $\Sigma_{ess}^{\mathbb{R}}$) is defined by the eigenvalues $\mu^{j}_{\pm}(\lambda)$ of the monodromy matrices $M_{\pm}(\lambda)$ (resp. $M_{0}(\lambda)$).
However,
most of the proof is unchanged.
Of course,
Theorem \ref{thm:1} holds even if the asymptotic constant cases and periodic case are mixed.
 
Theorem \ref{thm:1} and Theorem \ref{thm:2} were proved by Sandstede and Scheel \cite{SS1} using analytical methods,
and Theorem \ref{thm:2} has been essentially proved by Gardner \cite{RG1}.
In the following sections,
we give topological proofs of these theorems using a generalization and extension of Nii's arguments in \cite{SN1}.

For Theorem \ref{thm:2},
we emphasize that our result holds even if $\gamma$ takes another value, because we only use the topological transversality of $U_{\pm}$ and $\bar{E}_{0}(\lambda)$.

%%%%%%%%%%%%%%%%%%%%%%%%%%%%%%%%%%%%%%%%%%%%%%%%%%%%%%%%%%%%%%%%%%%%%%%%%%

\subsection{Proof of Theorem 1}\label{sec:6}
We only show the accumulation of eigenvalues on $\Sigma^{+}_{abs}$.
Therefore, we assume that $S_{abs}^{D,+}=\Sigma_{abs}^{D,+} \cap S_{abs}^{D}$ is not the empty set.
The proof for the case of $\Sigma_{abs}^{-}$ is exactly the same if we take the backward orbit.

Let $\Phi(x,y;\lambda)$ be the fundamental solution matrix for \eqref{eq:1} defined by
\begin{equation*}
\frac{\partial}{\partial x}\Phi(x,y;\lambda)=A(x;\lambda)\Phi(x,y;\lambda), \ \ \Phi(x,x;\lambda)=I.
\end{equation*}

Since $\Phi(x,y;\lambda) \in GL_{N}(\mathbb{C})$,
$\Phi(x;\lambda)U_{-}$ is an $i_{-}$-dimensional subspace of $\mathbb{C}^{N}$ for fixed $x$ and $\lambda$,
and $U_{+}$ is an $i_{+}$-dimensional subspace.
The eigenvalue problem of ${\cal T}_{\ell}$ can then be rewritten as follows.
\begin{lem}
$\lambda$ is an eigenvalue of ${\cal T}_{\ell}$ if and only if $\Phi(\ell-\ell;\lambda)U_{-} \cap U_{+} \neq \{ 0 \}$.
\end{lem}

The fundamental solution matrix $\Phi(x,y;\lambda)$ acts on any subspaces of $\mathbb{C}^{N}$.
Therefore it induces a flow on the Grassmann manifold $G_{k}(\mathbb{C}^{N})$ for any $k$.
We rewrite the eigenvalue problem of ${\cal T}_{\ell}$ with respect to the existence of specific connecting orbits.
Consider the system on $G_{i_{-}}(\mathbb{C}^{N})$ induced by \eqref{eq:1}
\begin{eqnarray}\label{eq:2}
G' = {\cal G}(x,G;\lambda), \ G \in G_{i_{-}}(\mathbb{C}^{N}),
\end{eqnarray}
where $i_{-}=\dim U_{-}$.
Since $U_{+}$ is an $i_{+}$-dimensional subspace of $\mathbb{C}^{N}$,
$U_{+}$ is a subset in $G_{i_{-}}(\mathbb{C}^{N})$.

Let $G(x;\lambda)$ be a solution of \eqref{eq:2} satisfying $G(-\ell;\lambda)=U_{-}$.
This corresponds to the projectivization $\mathbb{P}(\Phi(x,-\ell;\lambda)U_{-})$.
Define the right-side boundary condition of \eqref{eq:2} by
\begin{equation}\label{eq:3}
\hat{U}_{+}:=\{ G \in G_{i_{-}}(\mathbb{C}^{N}) \mid \dim(G \cap U_{+}) \geq 1 \}.
\end{equation}
In addition,
we define a subset ${\cal E}_{+}(\lambda) \subset G_{k}(\mathbb{C}^{N})$ as follows.

Let us set $V_{k-1}:= E^{1}_{+}(\lambda) \cap E^{2}_{+}(\lambda), \ V_{k+1}:= \bar{E}_{+}(\lambda)$,
and take the flag $V=(0 \subset V_{1}) \subset V_{2} \subset \cdots \subset V_{N}$ in $\mathbb{C}^{N}$.
We then set
\begin{equation}\label{eq:8}
{\cal E}_{+}(\lambda)=\{ G \in G_{k}(\mathbb{C}^{N}) \mid V_{1} \subset \cdots V_{k-1} \subset G \subset V_{k+1}=\bar{E}_{+}(\lambda) \}.
\end{equation}

$\hat{U}_{+}$ and ${\cal E}_{+}(\lambda)$ are called the Schubert cycles on $G_{i_{-}}(\mathbb{C}^{N})$ (see \cite{GPHJ}).
Using the Schubert calculus,
we have the following properties.
\begin{lem}\label{lem:trans}
$\hat{U}_{+}$ and ${\cal E}_{+}(\lambda)$ are $1$-codimensional and $1$-dimensional submanifolds in $G_{i_{-}}(\mathbb{C}^{N})$,
respectively.
Moreover,
$\hat{U}_{+}$ intersects ${\cal E}_{+}(\lambda)$ transversally,
and the intersection number is equal to $1$.
\end{lem}
\begin{proof}
For any flag $V=(V_{1} \subset V_{2} \subset \cdots \subset V_{N})$ in $\mathbb{C}^{N}$ and a sequence $a=(a_{1},\cdots,a_{k})$ where $N-k \geq a_ {1} \geq \cdots \geq a_{k} \geq 0$,
the Schubert cycle 
\begin{equation*}
\sigma_{a}(V) = \{ G \in G_{k}(\mathbb{C}^{N}) \mid \dim (V_{N-k+i-a_{i}}\cap G)\geq i \}
\end{equation*}
is homeomorphic to $\mathbb{C}^{k(N-k) - \sum a_{i}}$.
Therefore we have
\begin{equation*}
\hat{U}_{+} \cong \mathbb{C}^{k(N-k) - 1}, \ {\cal E}_{+}(\lambda) \cong \mathbb{C}^{k(N-k)-((k-1)(N-k)+(N-k-1))}.
\end{equation*}
Let $\sharp (\sigma_{a},\sigma_{b})$ be the intersection number of Schubert cycles $\sigma_{a}$ and $\sigma_{b}$.
For $\sigma_{a}$ and $\sigma_{b}$ where $\sum a_{i} + \sum b_{i} =k(N-k)$,
$b_{k-i+1}=N-k -a_{i}$ if and only if
$\sharp (\sigma_{a},\sigma_{b})=1$.
Consequently we have $\sharp (\hat{U}_{+},{\cal E}_{+}(\lambda))=1$.
\end{proof}

The eigenvalue problem of ${\cal T}_{\ell}$ can then be expressed as follows.
\begin{lem}\label{lem:1}
$\lambda$ is an eigenvalue of ${\cal T}_{\ell}$ if and only if $G(\ell;\lambda) \in \bar{U}_{+}$. 
\end{lem}
\begin{proof}
$\lambda$ is an eigenvalue of ${\cal T}_{\ell}$ if and only if $\Phi(\ell,-\ell;\lambda)U_{-}\cap U_{+} \neq \{ 0 \}$.
Therefore,
it is equivalent to the following projective condition
\begin{equation*}
\mathbb{P}(\Phi(\ell,-\ell;\lambda)U_{-})\cap \mathbb{P}(U_{+})\neq \emptyset.
\end{equation*}
This completes the proof.
\end{proof}

By Lemma \ref{lem:1},
the eigenvalue problem of ${\cal T}_{\ell}$ can be expressed as the existence problem of a connecting orbit from $U_{-}$ to $\hat{U}_{+}$ in $G(\ell;\lambda)$.
In other words,
it is the intersection problem between $G(\ell;\lambda)$ and $\hat{U}_{+}$.

To understand the accumulation of eigenvalues for the large bounded interval, 
we have to consider the asymptotic behavior of $G(\ell;\lambda)$ for $\lambda \in B(\lambda_{c};\delta)$ when $\ell \to \infty$. 
It arises as the increasing intersection number when $\lambda$ moves on the absolute spectrum $\Sigma_{abs}^{+}$.

To obtain topological information about the asymptotic behavior of $G(\ell;\lambda)$ and the intersection between $G(\ell;\lambda)$ and $\hat{U}_{+}$,
we use the Pl\"ucker embedding
\begin{equation}\label{eq:4}
\iota:G_{i_{+}}(\mathbb{C}^{N}) \to \mathbb{P}(\wedge^{i_{-}}\mathbb{C}^{N}) = \mathbb{CP}^{m-1},
\end{equation}
where $m =\begin{pmatrix}  N \\ i_{-} \end{pmatrix}$.
We then consider the system on the complex projective space $\mathbb{CP}^{m-1}$ induced by \eqref{eq:3}
\begin{equation}\label{eq:5}
P'={\cal P}(x,P;\lambda), \ P \in \mathbb{CP}^{m-1}.
\end{equation}
That is,
\eqref{eq:1} induces a linear system on the space of the $i_{-}$-form $\wedge^{i_{-}}\mathbb{C}^{N}$
\begin{equation}\label{eq:6}
Y^{(i_{-})'}=A^{(i_{-})}(x;\lambda)Y^{(i_{-})}, \ Y^{(i_{-})} \in \wedge^{i_{-}}\mathbb{C}^{N}.
\end{equation}
Then, \eqref{eq:5} is given by the projectivization of \eqref{eq:6}.
By Hypothesis \ref{hyp:1},
\eqref{eq:5} coincides with the autonomous system
\begin{equation}\label{eq:9}
P'={\cal P}_{+}(P;\lambda), \ P \in \mathbb{CP}^{m-1}
\end{equation}
if $\ell \geq \ell_{0}$,
which is induced by the linear autonomous system
\begin{equation}\label{eq:10}
Y^{(i_{-})'}=A_{+}^{(i_{-})}(\lambda)Y^{(i_{-})}, \ Y^{(i_{-})} \in \wedge^{i_{-}}\mathbb{C}^{N}.
\end{equation}
Note that $\wedge^{i_{-}}\mathbb{C}^{N} \cong \mathbb{C}^{m}$ and $A^{(i_{-})}=A \otimes I \otimes \cdots \otimes I + \cdots + I \otimes \cdots \otimes I \otimes A|_{\wedge^{i_{-}}\mathbb{C}^{N}}$.

Since $\Phi(x,-\ell;\lambda)U_{-}$ is an $i_{-}$-dimensional subspace in $\mathbb{C}^{N}$,
$\wedge^{i_{-}}\Phi(x,-\ell;\lambda)U_{-}$ is a $1$-dimensional subspace in $\wedge^{i_{-}} \mathbb{C}^{N}$.
Therefore,
the solution $G(x;\lambda)$ of \eqref{eq:3} corresponds uniquely to a solution $P(x;\lambda)=\iota(G(x;\lambda))$,
and corresponds to $\mathbb{P}(\wedge^{i_{-}} \Phi(x,-\ell;\lambda))U_{-}) \in \mathbb{CP}^{m-1}$.
Hence, the following proposition holds.
\begin{lem}\label{lem:2}
$\lambda$ is an eigenvalue of ${\cal T}_{\ell}$ if and only if $P(\ell;\lambda) \in \iota(\hat{U}_{+})$ in $\mathbb{CP}^{m-1}$.
\end{lem}

The asymptotic behavior of $P(\ell;\lambda)$ when $\ell \to \infty$ is determined by the eigenvalues of $A^{(i_{-})}_{+}(\lambda)$.
It is well known that the eigenvalues $\nu_{+}^{j}(\lambda), \ j=1,\cdots,m$ of $A^{(i_{-})}_{+}(\lambda)$ are given by the $i_{-}$ sum of eigenvalues $\mu_{+}^{k}(\lambda), \ k=1,\cdots,N$ of $A_{+}(\lambda)$.
Hence,
\begin{eqnarray*}
&&\nu_{+}^{1}(\lambda)= \mu_{+}^{1}(\lambda) + \cdots +  \mu_{+}^{i_{-}-1}(\lambda) + \mu_{+}^{i_{-}}(\lambda),\\
&&\nu_{+}^{2}(\lambda)= \mu_{+}^{1}(\lambda) + \cdots +  \mu_{+}^{i_{-}-1}(\lambda) + \mu_{+}^{i_{-} +1}(\lambda)
\end{eqnarray*}
has the following two properties.
\begin{lem}\label{lem:3}
For any $\lambda_{*} \in S_{abs}^{D,+}$,
the eigenvalues $\nu_{+}^{j}(\lambda), \ j=1,\cdots,m$ of $A^{(i_{-})}_{+}(\lambda)$ satisfy
\begin{equation*}
\real \nu_{+}^{1}(\lambda_{*}) = \real \nu_{+}^{2}(\lambda_{*}) > \real \nu_{+}^{j}(\lambda_{*}) , \ j=3,\cdots,m,
\end{equation*}
and
\begin{equation*}
\frac{d}{d\lambda} \left( \nu_{+}^{1}(\lambda) - \nu_{+}^{2}(\lambda) \right) \Big{|}_{\lambda = \lambda_{*}} \neq 0.
\end{equation*}
\end{lem}
\begin{lem}\label{lem:4}
Let $B(\lambda_{c};\delta)$ be the $\delta$-ball holding property \ref{property}.
Then $B(\lambda_{c};\delta) \setminus S_{abs}^{D,+} = B_{1} \cup B_{2}$ has similar property.
\begin{equation*}
\real \nu_{+}^{1}(\lambda) > \real \nu_{+}^{2}(\lambda) > \real \nu_{+}^{j}(\lambda) , \ j=3,\cdots,m,
\end{equation*}
for any $\lambda \in B_{1}$ and 
\begin{equation*}
\real \nu_{+}^{2}(\lambda) > \real \nu_{+}^{1}(\lambda) > \real \nu_{+}^{j}(\lambda) , \ j=3,\cdots,m,
\end{equation*}
for any $\lambda \in B_{2}$.
Moreover,
\begin{equation*}
\frac{d}{d\lambda}\left( \nu_{+}^{1}(\lambda) - \nu_{+}^{2}(\lambda) \right) \Big{|}_{\lambda = \lambda_{*}} \neq i r
\end{equation*}
for any $\lambda_{*} \in B(\lambda_{c};\delta) \setminus S_{abs}^{D,+}$ and $r \in \mathbb{R}$.
\end{lem}

The strategy is to find eigenvalues in $B(\lambda_{c};\delta)$.
First, we change to appropriate coordinates.
\begin{lem}\label{lem:5}
If $B(\lambda_{c};\delta)$ is sufficiently small,
there is a coordinate change that is analytic in $\lambda$ such that $A_{+}^{(i_{-})}(\lambda)$ forms
\begin{equation*}
A_{+}^{(i_{-})}(\lambda) =
\begin{pmatrix}
{\cal A} & O_{2,m-2}\\
O_{m-2,2} & {\cal B}(\lambda)
\end{pmatrix}
\end{equation*}
where 
${\cal A}_{+}(\lambda) = \diag \{ \nu_{+}^{1}(\lambda) , \nu_{+}^{2}(\lambda) \}$,
${\cal B}_{+}(\lambda)$ is a square matrix of dimension $m-2$,
and $O_{i,j}$ is the $i \times j$ zero matrix.
\end{lem}
\begin{proof}
This is easily seen from the fact that $\mathbb{C}^{m} = W_{1}(\lambda) \oplus W_{2}(\lambda) \oplus W_{s}(\lambda)$,
where $W_{j}(\lambda)$ is an eigenspace of $\nu_{+}^{j}(\lambda), \ j=1,2$,
and $W_{s}(\lambda)$ is the generalized eigenspace associated with $\nu_{+}^{j}(\lambda), \ j=3,\cdots,m$.
\end{proof}

For the coordinate $(Z_{1},\cdots,Z_{m})$ satisfying Lemma \ref{lem:5},
\eqref{eq:10} is written as
\begin{equation}\label{eq:7}
\left\{
\begin{split}
&Z_{u}' ={\cal A}_{+}(\lambda)Z_{u},\\
&Z_{s}'={\cal B}_{+}(\lambda)Z_{s},
\end{split}
\right.
\end{equation}
where $Z_{u}=(Z_{1},Z_{2})$ and $Z_{s}=(Z_{3},\cdots,Z_{m})$.

The following describes the extension and generalization of Nii's methods in \cite{SN1}.
Set $W_{u}(\lambda)=W_{1}(\lambda) + W_{2}(\lambda) = \wedge^{k}E^{1}_{+}(\lambda) + \wedge^{k}E^{2}_{+}(\lambda)$ and ${\cal SU}_{+} = \mathbb{P}(W_{u}(\lambda))$.
${\cal SU}_{+}(\lambda) = \{ [Z_{1}:Z_{2}:0:\cdots:0] \mid (Z_{1},Z_{2}) \neq 0 \}$ is identified with $\mathbb{CP}^{1}$,
and the following holds by Hypothesis \ref{hyp:3},
Lemma \ref{lem:trans}, and the injectivity of the Pl\"ucker embedding.
\begin{lem}\label{lem:6}
$\iota (\hat{U}_{+})$ and ${\cal SU}_{+}(\lambda)$ intersect transversally in the embedding manifold $\iota(G_{i_{-}}(\mathbb{C}^{N}))$.
\end{lem}
\begin{proof}
We show that ${\cal SU}_{+}(\lambda)=\iota ({\cal E}_{+}(\lambda))$.
For any $G \in {\cal SU}_{+}(\lambda)$ and $F \in \iota ({\cal E}_{+}(\lambda))$,
$F \wedge G =0$ because $\mathbb{P}(\wedge^{k}E^{1}_{+}(\lambda)), \mathbb{P}(\wedge^{k}E^{2}_{+}(\lambda))  \in {\cal SU}_{+}(\lambda) \cap \iota ({\cal E}_{+}(\lambda))$.
Therefore,
by the injectivity of the Pl\"ucker embedding and $\sharp(\hat{U}_{+}, {\cal E}_{+}(\lambda))=1$,
$\iota (\hat{U}_{+})$ and ${\cal SU}_{+}(\lambda)$ intersect transversally.
\end{proof}

To find intersections of $P(\ell;\lambda)$ and $\iota(\hat{U}_{+})$,
we consider the set $\Gamma := \{ P(\ell;\lambda) | \lambda \in B(\lambda_{c};\delta) \}$ and the projection
\begin{equation}
\pr:\mathbb{CP}^{m-1}\setminus \{ [0:0:Z_{3}:\cdots:Z_{m}] \mid  (Z_{3},\cdots,Z_{m}) \neq 0\} \to {\cal SU}_{+}(\lambda).
\end{equation}

We consider the asymptotic behavior of $P(\ell;\lambda)$ for $B(\lambda_{c};\delta)$.
Define the $\omega$-limit set of $P(\ell;\lambda)$ depending on $\lambda \in B(\lambda_{c};\delta)$ by
\begin{equation}
\Omega(\lambda):= \bigcap_{\bar{\ell} \in I_{\ell}} \clos \{ P(\ell;\lambda) \mid \ell > \bar{\ell}  \  \},
\end{equation}
where we then consider $\bar{\ell} < \ell <  \infty $.
\begin{lem}\label{lem:9}
$P(\ell;\lambda)$ satisfies the following asymptotic conditions.
\begin{eqnarray*}
\Omega(\lambda) =\left\{
\begin{array}{l}
 P_{n}=\mathbb{P}(\wedge^{i_{-}}E^{1}_{+}(\lambda)) \text{ for any } \lambda \in B_{1},\\ 
 P_{s}=\mathbb{P}(\wedge^{i_{-}}E^{2}_{+}(\lambda)) \text{ for any } \lambda \in B_{2}. 
 \end{array}\right.
\end{eqnarray*}
That is,
$P_{n}$ is an attracor and $P_{s}$ is a dual repeller in ${\cal SU}_{+}(\lambda)$ for any $\lambda \in B_{1}$,
and $P_{s}$ is an attracor and $P_{n}$ is a dual repeller in ${\cal SU}_{+}(\lambda)$ for any $\lambda \in B_{2}$.
Furthermore,
there is a periodic orbit $\gamma$ in ${\cal SU}_{+}(\lambda)$ for $Z_{u}' ={\cal A}_{+}(\lambda)Z_{u}$ such that
\begin{equation*}
\Omega(\lambda) = \gamma \ \text{ for any } \lambda \in S_{abs}^{D,+}.
\end{equation*}
\end{lem}
\begin{proof}
First,
we show that ${\cal SU}_{+}(\lambda)=\mathbb{P}(W_{s}(\lambda))$ is an attracting invariant space for \eqref{eq:9} with any $\lambda \in B(\lambda_{c};\delta)$.
For any $\lambda \in B(\lambda_{c};\delta)$,
the eigenvalues $\nu_{+}^{j}(\lambda)$ of $A_{+}^{i_{-}(\lambda)}$ satisfy
\begin{equation*}
\real \nu_{+}^{i}(\lambda) > \real \nu_{+}^{j}(\lambda), \ i=1,2, \text{ and } j=3,\cdots,m,
\end{equation*} 
and hence,
the solutions $Z_{s}(x;\lambda)=(Z_{1}(x;\lambda),Z_{2}(x;\lambda))$ and $Z_{u}(x;\lambda)=(Z_{3}(x;\lambda),\cdots,Z_{m}(x;\lambda))$ for \eqref{eq:7} satisfy
\begin{equation*}
 \frac{Z_{j}(\ell;\lambda)}{Z_{i}(\ell;\lambda)} \to 0 \ \text{ as } \ell \to \infty
\end{equation*}
for any $i=1,2$ and $j=3,\cdots,m$.
This implies that ${\cal SU}_{+}(\lambda)=\mathbb{P}(W_{s}(\lambda))$ is an attracting invariant space for \eqref{eq:9}.

We take the inhomogeneous coordinate $(\frac{Z_{2}}{Z_{1}} \cdots , \frac{Z_{m}}{Z_{1}})$ in a neighborhood of $P_{n}=[1:0:\cdots:0] \in \mathbb{CP}^{m-1}$,
and set $\alpha(\lambda) = \real ( \nu_{+}^{1}(\lambda) - \nu_{+}^{2}(\lambda))$.
Eigenvalues of $d{\cal P}(P_{n};\lambda)$ are given by $\nu_{+}^{j}(\lambda) - \nu_{+}^{1}(\lambda), \ j=2,\cdots,m $.
Since $\real \nu_{+}^{1}(\lambda) > \real \nu_{+}^{j}(\lambda)$ for any $\lambda \in B_{1}$, 
we have $\real (\nu_{+}^{j}(\lambda) - \nu_{+}^{1}(\lambda)) < 0, \ (2 \leq j \leq m)$,
and hence,
$P_{n}$ is a stable equilibrium point when $\lambda \in B_{1}$.
In addition,
we take the inhomogeneous coordinate $(\frac{Z_{1}}{Z_{2}} \cdots , \frac{Z_{m}}{Z_{2}})$ in a neighborhood of $P_{s}=[0:1:0:\cdots:0] \in \mathbb{CP}^{m-1}$.
Any solutions for \eqref{eq:9} restricted to ${\cal SU}_{+}(\lambda)$ are then controlled by
\begin{eqnarray*}
 \left(
 \frac{Z_{1}}{Z_{2}}
 \right)' =(\alpha(\lambda)+ i \omega))(\lambda) \left(
 \frac{Z_{1}}{Z_{2}}
 \right).
\end{eqnarray*}
Then $\alpha(\lambda) > 0$ and $\real (\nu_{+}^{j}(\lambda) - \nu_{+}^{2}(\lambda)) < 0, \ (3 \leq j \leq m)$.
This implies that $P_{s}$ is a repelling equilibrium point in ${\cal SU}_{+}(\lambda)$ when $\lambda \in B_{1}$.
Similarly,
$P_{s}$ is a stable equilibrium point and $P_{n}$ is a repelling equilibrium point in ${\cal SU}_{+}(\lambda)$ for any $\lambda \in B_{2}$.

Set $\omega(\lambda) = \imag (\nu_{+}^{1}(\lambda) - \nu_{+}^{2}(\lambda))$,
 and take the inhomogeneous coordinate $(\frac{Z_{1}}{Z_{2}})$ in ${\cal SU}_{+}(\lambda) \setminus \{P_{n} \cup P_{s} \}$.
 Any solutions for \eqref{eq:9} restricted to ${\cal SU}_{+}(\lambda)$ are then controlled by
 \begin{equation*}
 \left(
 \frac{Z_{1}}{Z_{2}}
 \right)' =i \omega(\lambda) \left(
 \frac{Z_{1}}{Z_{2}}
 \right),
 \end{equation*}
for any $\lambda \in S_{abs}^{D,+}$.
Therefore,
all solutions for \eqref{eq:9} consist of periodic solutions when $\lambda \in S_{abs}^{D,+}$. 
Since ${\cal SU}_{+}(\lambda)$ is an attracting invariant subspace,
the latter claim holds.
\end{proof}

The following proposition was proved by Nii \cite{SN1} for $i_{-}=1$ and $N=3$.
However,
it can be extended to the general case of $i_{+}$, because we have prepared the situation to agree with that in \cite{SN1}.
\begin{prop}\label{prop:nii1} \cite{SN1}.
For any $[Z_{1}:Z_{2}] \in \iota(\Gamma)$, there exists a neighborhood ${\cal N}$ such that each component of $\Gamma \cap \iota^{-1}({\cal N})$ is diffeomorphic to ${\cal N}$ and $C^{1}$-near to ${\cal N}$.
\end{prop}
\begin{proof}
Consider two points $\lambda_{1}, \lambda_{2} \in \clos( B(\lambda_{c};\delta))$ that are near to each other.
By Lemma \ref{lem:6},
we can take small neighborhoods $N_{n}$ and $N_{s}$ of $P_{n}$ and $P_{s}$ satisfying the following conditions:
\begin{eqnarray*}
&&N_{n} \cap N_{s}=\emptyset,\\
&&\Pr (P(\ell;\lambda)) \notin N_{s} \text{ for any } \lambda \in B_{1},\\
&&\Pr (P(\ell;\lambda)) \notin N_{n} \text{ for any } \lambda \in B_{2}
\end{eqnarray*}
if $\ell \geq \ell_{0}$ is sufficiently large.
In particular,
we choose $N_{n}$ and $N_{s}$ to satisfy ${\cal U}(\lambda)=\iota(\hat{U}_{+}\cap {\cal E}_{+}(\lambda)) \in {\cal SU}_{+}(\lambda) \setminus ( \hat{N}_{n} \cap \hat{N}_{s})$,
where $\hat{N}_{*} = {\cal SU}_{+}(\lambda) \cap N_{*}, \ *=n,s$.
Assume that $\lambda_{1}, \lambda_{2} \in B_{1}$.
Then,
$P(\ell;\lambda_{1})$ and $P(\ell;\lambda_{2})$ can be assumed to be outside $N_{s}$ and $N_{n}$,
respectively.
There exists $\bar{\ell}$ such that $P(\ell;\lambda_{1}) \neq P(\ell;\lambda_{2})$ for any $\ell \geq \bar{\ell}$,
because $\nu_{+}^{1}(\lambda) \neq \nu_{+}^{2}(\lambda)$ by the definition of $S_{abs}^{D,+}$.
The same argument holds for $\lambda_{1}, \lambda_{2}$.

If $\lambda_{1} \in B_{1}$ and $\lambda_{2} \in B_{2}$,
then $\Pr(P(\ell;\lambda_{1})) \to [1:0]$ and $\Pr(P(\ell;\lambda_{2})) \to [0:1]$ as $\ell \to \infty$ by Lemma \ref{lem:6}.
Hence,
$\Pr(P(\ell;\lambda_{1}))\neq \Pr(P(\ell;\lambda_{2})$ if $\ell$ is sufficiently large.
Therefore, $\Pr|_{\Gamma}$ is a local diffeomorphism if $\ell$ is sufficiently large.

By Lemma \ref{lem:9},
${\cal SU}_{+}(\lambda)$ is the attracting invariant subspace. Hence,
$\Gamma$ converges to ${\cal SU}_{+}(\lambda)$ as $\ell \to \infty$.
Therefore,
$\Gamma$ is locally $C^{1}$-near to ${\cal SU}_{+}(\lambda)$.
\end{proof}

Define a map by
\begin{equation*}
\psi_{\ell}=\Pr\circ \iota \circ G(\ell;\cdot):B(\lambda_{c};\delta) \to {\cal SU}_{+}(\lambda)/(N_{n} \cup N_{s}).
\end{equation*}
If we take $N_{n}$ and $N_{s}$ satisfying the conditions in Proposition \ref{prop:nii1},
then ${\cal U}(\lambda) = \iota (\hat{U}_{+} \cap {\cal E}_{+}(\lambda)) \in  {\cal SU}_{+}(\lambda) \setminus (N_{n} \cup N_{s})$.
Therefore, the following lemma holds.
\begin{lem}
$\lambda$ is an eigenvalue of ${\cal T}_{\lambda}$ if $\psi_{\ell}(\lambda) = {\cal U}(\lambda)$.
\end{lem}
\begin{proof}
$\psi_{\ell}(\lambda) = {\cal U}(\lambda)$ means that $\Gamma$ intersects with $\hat{U}_{+}$ at least once.
\end{proof}

The following proposition is an essential part of completing the proof.
\begin{prop}  
For any positive integer $n \in \mathbb{N}$,
there exists $\ell_{*}\geq \ell_{0}$ such that a map $\psi_{\ell}$ is an $n$-covering map for any $\ell\geq \ell_{*}$.
That is, $\psi_{\ell}(B(\lambda_{c};\delta)) = \iota(\Gamma)$ covers ${\cal SU}_{+}(\lambda)/(N_{n} \cup N_{s})$ more than $n$ times.
In particular,
$\psi_{\ell}({\cal C})$ covers $S^{1} \subset {\cal SU}_{+}(\lambda)/(N_{n} \cup N_{s})$ more than $n$ times where ${\cal C} = B(\lambda_{c};\delta) \cap S_{abs}^{D,+}$.
\end{prop}
\begin{proof}
First, we show the latter claim.
By the proof of Lemma \ref{lem:6},
any solutions for \eqref{eq:9} on ${\cal SU}_{+}(\lambda)$ are given by
\begin{equation}
\zeta(\zeta_{0},\ell;\lambda)=\zeta_{0}e^{i\omega(\lambda)\bar{\ell}},
\end{equation}
where $\zeta_{0} \in {\cal SU}_{+}(\lambda) \setminus\{ P_{n} ,P_{s} \}$,
$\bar{\ell}=\ell-\ell_{0}$ and $\lambda \in {\cal C}$.
We set $\bar{{\cal C}}= \clos (B(\lambda_{c};\delta)) \cap S_{abs}^{D,+}$ and choose $\lambda_{0}$ and $\lambda_{1}$ to be endpoints of $\bar{{\cal C}}$.
Consider a path $l:[0,1] \ni s \mapsto l(s) \in \bar{{\cal C}}$ with an initial point $l(0)=\lambda_{0}$ and an end point $l(1)=\lambda_{1}$.
Since $\zeta_{0}e^{i\omega(\lambda)\bar{\ell}}=\zeta_{0}e^{i2\pi\omega(\lambda)\bar{\ell}}$,
$\zeta(\zeta_{0},\ell;\bar{{\cal C}})$ moves on $S^{1}$ for any $\zeta_{0} \in {\cal SU}_{+}(\lambda) \setminus\{ P_{n} ,P_{s} \}$.
Therefore,
a curve $\zeta(\zeta_{0},\ell;l(s))$ moves on $S^{1}$ monotonically because $\omega(l(s))$ is monotone.
This implies that for any $n \in \mathbb{N}$,
there exists $\ell> \ell_{0}$ so that $|(\omega(\lambda_{1})-\omega(\lambda_{2}))\bar{\ell} | > n\pi$.
Then $\zeta(\zeta_{0},\ell;\bar{{\cal C}})$ moves on $S^{1}$ at least $n$ times.

On the other hand,
we have $\psi_{\ell}(\lambda) \to [1:0]$ with $\lambda \in B_{1}$ and  $\psi_{\ell}(\lambda) \to [0:1]$ with $\lambda \in B_{2}$ as $\ell \to \infty$.
Therefore $\psi_{\ell}(B_{1})$ covers the upside of $\psi_{\ell}({\cal C})$ in ${\cal SU}_{+}(\lambda)/(N_{n} \cup N_{s})$ at least $n$ times.
Similarly,
$\psi_{\ell}(B_{2})$ covers the downside of $\psi_{\ell}({\cal C})$ at least $n$ times.
\end{proof}

Therefore,
we can take $\ell > \ell_{0}$ satisfying $\psi_{\ell}^{-1}({\cal U})  =\{ \lambda_{1},\cdots,\lambda_{n}\}$,
and then take disjoint neighborhoods ${\cal R}_{i}$ of $\lambda_{i}, \ (i=1,\cdots,n)$ for which $\psi_{\ell}|_{{\cal R}_{i}}$ is injective and ${\cal U}(\lambda) \in \iota(\Gamma_{i})$,
where $\Gamma_{i}:=\{ P(\ell;\lambda) \mid \lambda \in {\cal N}_{i} \}$.
$\Gamma_{i}$ and $\iota(\hat{U}_{+})$ intersect transversally for each $\lambda \in B(\lambda_{c};\delta)$ because $\iota(\hat{U}_{+})$ and $\iota({\cal E}_{+}(\lambda))={\cal SU}_{+}(\lambda)$ intersect transversally,
and ${\cal SU}_{+}(\lambda)$ and $\Gamma$ are $C^{1}$-near.

We then denote the intersection point $\Gamma_{i} \cap \iota(\hat{U}_{+})$ as $g(\lambda)$ and define a map
\begin{equation*}
\Lambda_{i}:\Gamma_{i} \ni P(\ell;\lambda) \mapsto g(\lambda) \in \Gamma_{i}.
\end{equation*}
The following is obtained from Brouwer's fixed point theorem.
(These arguments are exactly the same as those given by Nii \cite{SN1}.)
\begin{lem}\label{lem:7}(\cite{SN1}).

$\Lambda_{i}$  has a fixed point.
\end{lem}

\begin{lem}\label{lem:8} \cite{SN1}.
Let $\lambda^{*} \in {\cal R}_{i}$ be the fixed point of $\Lambda_{i}$.
Then, $\lambda^{*}$ is an eigenvalue of ${\cal T}_{\ell}$.
\end{lem}

%%%%%%%%%%%%%%%%%%%%%%%%%%%%%%%%%%%%%%%%%%%%%%%%%%%%%%%%%%%%%%%%%%%%%%%%
\subsection{Proof of the Theorem 2}\label{sec:7}
In this case,
we only show that the non-degenerate essential spectrum $S_{per}$ plays the same role as the non-degenerate absolute spectrum $S_{abs}^{\Omega}$ for the first-order system on the embedded Grassmannian manifold.

In the periodic boundary conditions,
$\lambda$ is an eigenvalue of ${\cal T}_{\ell}^{per}$ if and only if there exists a nontrivial solution ${\hat Y}(x;\lambda)$ of \eqref{eq:p1} satisfying the separated boundary conditions,
\begin{equation}
{\hat Y}(-\ell;\lambda) \in U_{-}, \ {\hat Y}(\ell;\lambda) \in U_{+},
\end{equation} 
where $U_{\pm}=\{ (Y,Y) \mid Y \in \mathbb{C}^{N} \}$ are boundary subspaces for \eqref{eq:p1}.
Therefore the following lemma holds.
\begin{lem}
Let $\Phi(x,y;\lambda)$ be a fundamental solution matrix for \eqref{eq:1}.
Then, $\hat{\Phi}(x,y;\lambda):=\Phi(x,y;\lambda) \oplus I$ is a fundamental solution matrix for \eqref{eq:p1} that is defined by
\begin{equation*}
\begin{split}
&\hat{\Phi}(x,y;\lambda):\mathbb{C}^{N} \oplus \mathbb{C}^{N} \to \mathbb{C}^{N} \oplus \mathbb{C}^{N},\\
&\hat{\Phi}(x,y;\lambda)\hat{Y}= (\Phi(x,y;\lambda)Y,W) 
\end{split}
\end{equation*}
for any $\hat{Y}=(Y,W) \in \mathbb{C}^{2N}$.
\end{lem}

Therefore we obtain the topological invariant for periodic boundary conditions via separated boundary conditions.
Note that the dimension of $U_{\pm}$ coincides with $N$.

This leads to the following lemma.
\begin{lem}\label{lem:p1}
$\lambda$ is an eigenvalue of ${\cal T}_{\ell}^{per}$ if and only if $\hat{\Phi}(\ell,-\ell;\lambda)U_{-} \cap U_{+}\neq \{ 0 \}$,
where $\hat{\Phi}(x,y;\lambda)$ is the fundamental solution matrix for \eqref{eq:p1}.
\end{lem}

The strategy of the proof is as follows. 
Denote $V(x;\lambda):=\Phi(x,\ell;\lambda)U_{-} \cap \bar{E}_{0}(\lambda)$.
To find the intersection of $\Phi(\ell,\ell;\lambda)U_{-}$ and $U_{+}$,
we track the one-dimensional subspace $V(\ell;\lambda)$ in the complex projective space $\mathbb{CP}^{N}=\mathbb{P}(\bar{E}_{0}(\lambda))$.

Let $E^{c}(\lambda)$ be an $N$-dimensional generalized eigenspace associated with $ \overbrace{ 0,\cdots,0}^{N}$ and $E^{k}(\lambda)$ be an one-dimensional eigenspace associated with $\mu_{0}^{k}(\lambda)$.
Denote ${\cal U}(\lambda):= \mathbb{P}(U_{+}\cap \bar{E}_{0}(\lambda)) \in {\cal E}_{0}(\lambda)$ and ${\cal E}_{0}(\lambda) := \mathbb{P}(\bar{E}_{0}(\lambda))/\mathbb{P}(E^{c}(\lambda))$ where ${\cal E}_{0}(\lambda) \cong \mathbb{CP}^{1}$.
Put $P(\ell;\lambda)=\mathbb{P}(V(\ell;\lambda))$,
and $\hat{P}(\ell;\lambda)$ is a corresponding element in ${\cal E}_{0}(\lambda)$.
By Hypothesis \ref{hyp:p3},
$\hat{P}(\ell;\lambda),\ {\cal U}(\lambda) \in {\cal E}_{0}(\lambda)$.
Then the following lemma holds.
\begin{lem}\label{lem:p2}
$\lambda$ is an eigenvalue of ${\cal T}_{\ell}^{per}$ if $\hat{P}(\ell;\lambda) = {\cal U}(\lambda)$.
\end{lem}

Consider the asymptotic behavior of $\hat{P}(\ell;\lambda)$ for $B(\lambda_{c};\delta)$.
A restricted system of \eqref{eq:p1} in $\bar{E}_{0}(\lambda)$ is given by
\begin{eqnarray*}
Z'=\hat{A}_{0}(\lambda)Z
\end{eqnarray*}
where $Z=(Z_{1},\cdots,Z_{N+1}) \in \bar{E}_{0}(\lambda)$ and $\hat{A}_{0}(\lambda)=\diag\{ \mu_{0}^{k}(\lambda),\overbrace{ 0 , \cdots, 0}^{N} \}$.

We take the inhomogeneous coordinate $(\frac{z_{2}}{z_{1}})$ in a neighborhood of $P_{n}=\mathbb{P}(E^{k}(\lambda)) \in {\cal E}_{0}(\lambda)$.
$\hat{P}(\ell;\lambda)$ is then controlled by
\begin{equation*}
\Bigl(\frac{z_{2}}{z_{1}}\Bigr)' = -\mu_{0}^{k}(\lambda)(\frac{z_{2}}{z_{1}}\Bigr).
\end{equation*}
Similarly,
we take the inhomogeneous coordinate $(\frac{z_{1}}{z_{2}})$ in a neighborhood of $P_{s}=[\mathbb{P}(E^{c}(\lambda))] \in {\cal E}_{0}(\lambda)$.
$\hat{P}(\ell;\lambda)$ is then controlled by
\begin{equation*}
\Bigl(\frac{z_{1}}{z_{2}}\Bigr)' = \mu_{0}^{k}(\lambda)(\frac{z_{1}}{z_{2}}\Bigr).
\end{equation*}
Since $\real \mu_{0}^{k}(\lambda) > 0$ for any $\lambda \in B_{1}$ and $\real \mu_{0}^{k}(\lambda) < 0$ for any $\lambda \in B_{2}$,
the following lemma holds.
\begin{lem}\label{lem:p3}
$P_{n}$ is an attractor and $P_{s}$ is a dual repeller in ${\cal E}_{0}(\lambda)$ for any $\lambda \in B_{1}$,
and $P_{s}$ is an attractor and $P_{n}$ is a dual repeller in ${\cal E}_{0}(\lambda)$ for any $\lambda \in B_{2}$.
\end{lem}

Let $N_{n}$ and $N_{s}$ be neighborhoods of $P_{n}$ and $P_{s}$ satisfying $N_{n} \cap N_{s} = \emptyset$ and ${\cal U}(\lambda) \in {\cal E}_{0}(\lambda) \setminus \{ N_{n} \cup  N_{s}\}$.
We then define the map
\begin{equation*}
\psi_{\ell}=\hat{P}(\ell;\cdot): B(\lambda_{c};\delta) \to {\cal E}_{0}(\lambda) /  \{ N_{n} \cup  N_{s}\}.
\end{equation*}

The remainder of the proof uses the same arguments as for the case of separated boundary conditions.
That is,
for any positive integer $n \in \mathbb{N}$,
there exists $\ell_{*}>0$ such that a map $\psi_{\ell}$ is an $n$-covering map for any $\ell \geq \ell_{*}$ and hence,
$\{ {\cal T}_{\ell}^{per} \}$ has at least $n$ eigenvalues in $B(\lambda_{c};\delta)$.

%%%%%%%%%%%%%%%%%%%%%%%%%%%%%%%%%%%%%%%%%%%%%%%%%%%%%%%%%%%%%%%%%%%%%%%%%%%

\section{Acknowledgments}
The author would like to thank Yasumasa Nishiura for many important suggestions. 
He also would like to thank Shunsaku Nii for providing useful comments and stimulating discussions.
He gratefully acknowledges helpful discussions with Shin-Ichiro Ei and Takashi Teramoto on several points in the paper.
This work was supported by CREST, JST.

\end{document}